\documentclass{amsart}
\usepackage{graphicx}
\usepackage[utf8]{inputenc}
\usepackage[T1]{fontenc}
\usepackage{textcomp}
\usepackage[english]{babel}
\usepackage{url}
\usepackage{amsmath,amssymb,amsthm}
\usepackage{mathrsfs}

\usepackage{import}
\usepackage[dvipsnames]{xcolor}
\usepackage[colorlinks=true,linktoc=page]{hyperref}
\hypersetup{urlcolor=NavyBlue, citecolor=red, linkcolor=RubineRed}

\usepackage[all,cmtip]{xy}
\usepackage{manfnt}

\usepackage[backend=biber, style=alphabetic, natbib=true,hyperref=true]{biblatex}
\theoremstyle{plain}
\newtheorem{Theo}{Theorem}[section] 
\newtheorem{Pro}[Theo]{Proposition}        
\newtheorem{Lem}[Theo]{Lemma}            
\newtheorem{cor}[Theo]{Corollary}

\theoremstyle{definition}
\newtheorem{Defi}[Theo]{Definition}

\newtheorem{nota}[Theo]{Notation}

\newtheorem{hyp}[Theo]{Convention}

\theoremstyle{remark}
\newtheorem{rem}[Theo]{Remark}

\def\ogg~{{\rm \og}}   

\def\nn{\noindent}

\def\qq{\nn\quad}
\def\qqq{\nn\quad\quad}

\def\emptyset{\varnothing}

\def\NN{{\mathbb N}}    
\def\ZZ{{\mathbb Z}}     
\def\RR{{\mathbb R}}    
\def\AA{{\mathbb A}}    

   \def\cM{{\mathcal M}}          \def\cO{{\mathcal O}}              \def\cL{{\mathcal L}}     

\newcommand\ct{\operatorname{cotan}}         

\newcommand{\disf}[2]{D^+(#1,#2)}
\newcommand{\diso}[2]{D^-(#1,#2)}
\newcommand{\fdisf}[2]{\cO(D^+(#1,#2))}

\newcommand{\h}[1]{\mathscr{H} (#1)}
\newcommand{\A}[1]{\AA^{1,\mathrm{an}}_{#1}}

\newcommand{\D}[1]{\frac{\mathrm{d}}{\mathrm{d#1}}}
\newcommand{\Ct}[1]{\hat{\otimes}_{#1}}
\newcommand{\LL}[2]{\cL_{#1}(#2)}

\newcommand{\Lk}[1]{\LL{k}{#1}}
\newcommand{\nsp}[1]{\rVert #1\rVert_{\mathrm{Sp}}}

\newcommand{\piro}[2]{\pi_{#1/#2}}
\newcommand{\pik}[1]{\piro{#1}{k}}

\newcommand{\nor}[1]{\rVert #1\rVert}

\newcommand{\discf}[3]{D^+_{#1} (#2,#3)}
\newcommand{\disco}[3]{D^-_{#1} (#2,#3)}

\newcommand{\Couf}[4]{C^+_{#1} (#2,#3,#4)}
\newcommand{\Couo}[4]{C^-_{#1} (#2,#3,#4)}
\newcommand{\fCouf}[4]{\cO(C^+_{#1} (#2,#3,#4))}

\newcommand{\Fonction}[4]{\begin{array}[c]{rcl} 
                            #1&\longrightarrow&#2\\ #3&\mapsto&
                                                                #4\\ \end{array}}

\newcommand{\fra}[1]{\frac{1}{#1}}
\newcommand{\Dmod}[2]{#1-\text{\bf Mod}(#2)}

\renewcommand\phi{\varphi}
\renewcommand\epsilon{\varepsilon}
\def\ct{\hat{\otimes}}

\def \ot{\otimes}
\def \<{\langle}
\def \>{\rangle}

\def\|{\rVert}
\def\*{\blacklozenge}

\def\mr{\mathrm}

\def\dT{\frac{\mr{d}}{\mr{dT}}}
\def\d{\frac{\mr{d}}{\mr{dS}}}

\def\dz{\D{Z}}

\def \R+{\RR_{+}}

\def\Ak{\A{k}}

\def \Ao{\A{\Omega}}

\def \-1{^{-1}}

\def \Hx{\mathscr{H}(x)}
\def \Hy{\mathscr{H}(y)}

\def\Bigcup{\bigcup\limits}

\def \Sum{\sum\limits}

\def\Bigoplus{\bigoplus\limits}

\def\Min{\min\limits}

\def\(({(\!(}
\def\)){)\!)}
\def\DD{\mathscr{D}}

\def \KS{k\(( S\))}
\def \KT{k\(( T\))}
\def \ks{k[\![S]\!]}
\def\diffk{\Dmod{S\d}{\KS}}

\numberwithin{equation}{section}

\hypersetup{
pdfauthor = {Tinhinane Amina AZZOUZ},
pdfsubject = {p-adic differential equations},
pdftitle = {Spectrum of a linear differential equation over a field of
  formal power series},
pdfkeywords = {formal differential equations, spectral theory},
}

\begin{document}

\title{Spectrum of a linear differential equation over a field of
  formal power series}

\author{Tinhinane A. AZZOUZ}
\address{Univ. Grenoble Alpes, CNRS, Institut Fourier, F-38000 Grenoble, France}
\curraddr{Université Alger 1, 02 Rue Didouche Mourad, Algiers, Algeria}
\email{t.azzouz@univ-alger.dz}



\begin{abstract}
  In this paper we associate to a linear differential equation with
coefficients in the field of Laurent formal power series a new
geometric object, a spectrum in the sense of Berkovich. We compute this
spectrum and show that it contains interesting informations about the equation. 
\end{abstract}
\maketitle
\section*{Introduction}
Over the centuries, the theory of differential equations represents an
important field of
mathematics. Notably, in the real and complex context, they
encode behaviours of many physical phenomena. However, in
ultrametric setting, they are rather related to algebraic and number theory problems. This part of the theory, appears around 1960, since Dwork's work on the variation of the Zeta function, and has become a central subject of
investigation.

The study of ultrametric differential equations is more
complicated than the usual context even for the linear case. Indeed,
the solutions of an ultrametric linear differential equation may have 
finite radii of convergence, even without the presence of poles.

However, these radius behaves in a very controlled way, and their
knowledge permits to obtain several informations about the
equation. In particular, under some assumption, if a differential
equation has two solutions with different radii, then it should
correspond to a decomposition of the equation. This was firstly
introduced by Dwork and mainly developed in  the works of Robba \cite{Rob75a},\cite{Rob75b},\cite{Rob80},
Dwork an Robba \cite{DR77}, Christol and Mebkhout\cite{CM00},
\cite{CM01}, Kedlaya\cite{Ked}, \cite{ked13}. We point out that a large part of the
literature is devoted to the two special cases: differential
equation over a germ of a punctured disk, or Robba ring. In the
special case of a trivial valued field $k$ of characteristic zero, the Robba ring corresponds
exactly to the field of Laurent formal power series $\KT$, and
the radii of convergence of a differential equation over $\KT$ are
strongly related to the  {\it formal} slopes (see for example \cite[Section~4.3]{and}). The main fundamental
decomposition and classification are expressed in the following way: Consider the differential field
$(k(\!(T)\!),T\dT)$. We mean by a differential equation with
coefficients in $k(\!(T)\!)$ a differential module $(M,\nabla)$ over
$(k(\!(T)\!),T\dT)$.
\begin{enumerate}
\item {\em Decomposition theorem according to the slopes}. Considering
  the $T$-adic valuation of the coefficients of the operator $\nabla$ in a
  cyclic basis,  we can associate a Newton polygon, called formal. The
  formal slopes of $(M,\nabla)$ are the slopes of this polygon. The
  decomposition theorem is the following. 
  \begin{Theo}[{\cite[p. 97-107]{DMR07}}]
    Let $\gamma_1<\cdots<\gamma_\mu$ be the slopes of the formal
    Newton polygon
    of $(M,\nabla)$, with multiplicity  $n_1,\cdots,
    n_\mu$ respectively. Then \[(M,\nabla)=\bigoplus_{i=1}^\mu(M_{\gamma_i},\nabla_{\gamma_i}),\]
    where $M_{\gamma_i}$ has dimension $n_i$ and  a unique slope $\gamma_i$ with multiplicity $n_i$.
  \end{Theo}
\item {\em The Turrittin-Levelt-Hukuhara decomposition
    theorem} \cite{Kat70}, \cite{Rob80}, \cite{Ked}. It
  claims that for any differential module $(M,\nabla)$, there exists a
  suitable finite extension $k'(\!(T^{\fra{n}})\!)$ of $k(\!(T)\!)$
  for which the pull-back of $(M,\nabla)$ with respect to this
  extension is an extension of differential modules of rank
  one.
\end{enumerate}

This paper is a continuity of our work \cite{Cons}, where we
introduce, a new geometric invariante, 
the spectrum in the sense of Berkovich of a differential
module, and develop some material for the computation of this one.

In this paper we focus on the computation of the spectrum of a
differential module over $(\KT,T\dT)$, we show that it contains
intersting information. More precisely, we can
recover in one hand all the formal slopes of the differential module,
in the other hand the exponentes of the regular part
of the differential module. For this purpose we will use the
classification results listed above.
Before announcing the main
result of the paper, we shall recall quickly the notion of the spectrum in the
sense of Berkovich. 

Recall that for an element $f$ of a non-zero $k$-algebra $E$ with unit, the classical
spectrum of $f$ is the set
\[\{ a\in k;\; f-a.1_E\text{ is not invertible in } E\}.\]
This set may be empty, even if $E$ is a $k$-Banach algebra. To  deal
with this issue Berkovich proposed to consider the spectrum not as a subset
of $k$, but
as a subset of the analytic affine line $\Ak$ (which is a bigger space
than $k$)\cite[Chapter 7]{Ber}. Let $E$ be a non-zero $k$-Banach
algebra and $f\in E$. The spectrum $\Sigma_{f,k}(E)$ of $f$ in the sense of Berkovich is
the set of points of $\Ak$ that correspond to a pair $(\Omega,c)$,
where $\Omega$ is a complete extension of $k$ and $c\in \Omega$, for
which $f\ot 1-1\ot c$ is not invertible in $E\ct_k\Omega$
. This
spectrum is non-empty, compact and satisfies other nice properties
(cf.\cite[Theorem 7.1.2]{Ber}).

Let $(M,\nabla)$ be a differential module over
$(k(\!(T)\!),T\dT)$. Let $r\in (0,1)$ be a real positive number. From now on we endow $k(\!(T)\!)$
with the $T$-adic absolute value given by
\begin{equation}
  \label{eq:26}
  |\sum_{i\geq N} a_iT^{i}|:=r^{N},
\end{equation}
if $a_N\ne
  0$. In this setting $k(\!(T)\!)$ is a complete valued field and the induced valuation on $k$ is the trivial valuation. From
  now on we endow $k$ with the trivial valuation. We can endow $M$ 
  with $k(\!(T)\!)$-Banach structure, moreover it induces a $k$-Banach structure for which
  $\nabla:M\to M$
  is a bounded operator. As in our previous work \cite{Cons}, the
  spectrum $\Sigma_{\nabla,k}(\Lk{M})$ of $(M,\nabla)$ will be the spectrum of $\nabla$ as an
  element of the $k$-Banach algebra $\Lk{M}$ of bounded endomorphism of
  $M$ with respect to operator norm.

Notice that we cannot use the classical index theorem of B. Malgrange
  \cite{Malgrange-Irreg} to compute neither the spectrum in the sense
  of Berkovich nor the classical spectrum of $\nabla$. Indeed, it is relatively easy to show that any non trivial rank
  one connection on $k\(( T\))$ is set-theoretically bijective. However, the set-theoretical inverse of the
  connection may not be bounded. This is due to the fact
  that the base field $k$ is trivially valued and  Banach open
  mapping theorem does not hold in this setting.  

  For any positive real number $l$,
  we set $x_{0,l}$ to be the point
  of $\Ak$ associated to $l$-Gauss norm on $k[T]$
  (i.e. $\sum_ia_iT^i\mapsto \max_i|a_i|l^i$). The main result of the paper is the following:

  \begin{Theo}\label{sec:introduction}
  Assume that $k$ is algebraically closed. Let $r\in (0,1)$. Assume that $(\KS,|.|)\simeq (\h{x_{0,r}},|.|)$. Let $(M,\nabla)$ be a differential module over $(k(\!(T)\!),T\dT)$. Let
  $\{\gamma_1,\cdots, \gamma_{\nu}\}$ be the set of the slopes of
  $(M,\nabla)$ and let
  $\{a_1,\cdots, a_{\mu}\}$ be the set of the exponents of the regular
  part of $(M,\nabla)$. Then the spectrum of $\nabla$ as an element of
  $\Lk{M}$ is:
  \[\Sigma_{\nabla,k}(\Lk{M})=\{x_{0,r^{-\gamma_1}},\cdots,
    x_{0,r^{-\gamma_{\nu}}}\}\cup \bigcup_{i=1}^{\mu}(a_i+\ZZ).\]
\end{Theo}
This result clearly shows the importance of the points of
the spectrum that are not in $k$. Indeed, form these points we can
recover the slopes without multiplicity of the differential module. Therefore, we can say that the spectrum is highly
connected not only to the smallest radius of convergence, which was
the first motivation of our work, but to all radii of convergence of
the solutions of a linear differential equation. 

On other hand, although differential modules over $(\KS,S\d)$
  are algebraic objects, their spectra in the sense of Berkovich depend highly on the choice of
  the absolute value on $\KS$.

  The paper is organized as follows.  Section~\ref{sec:defin-an-notat}
  is devoted to providing setting and notations. More precisely, we
  will recall the definition of the analytic
  affine line $\Ak$, and give, in the setting of the paper, a precise topological description for
  disks and annuli of $\Ak$. We will recall also the definition of the spectrum in the sense of
  Berkovich.

  In Section~\ref{sec:diff-modul-over}, we introduce the spectrum of
  a differential module and recall some properties given in
  \cite{Cons}. We also explain the behavior of this spectrum after ramified ground field extension. In the end of the
  section we recall the definition of the formal Newton polygon.

Section~\ref{sec:spectr-diff-module-5} is devoted to announcing and proving the main result of
  the paper. Using the decomposition theorem according to the slopes,
  we can reduce the probleme to the computation of the spectrum of 
  regular singular differential modules and differential modules without
  regular part. Since \cite[Proposition~3.15]{Cons}, in order to
  compute the spectrum of a regular singular
  module, it is enough to compute the spectrum of
  $T\dT$ as an element of $\Lk{k(\!(T)\!)}$. For differential modules without regular part, we reduce the computation of the spectrum to the case of differential modules of rank one. This is possible by using
  Turrittin-Levelt-Hukura decomposition theorem and the behaviour of
  the spectrum after ramified ground field extension.

  In the last section, we give a short discussion about the result and
  explain more precisely how it connect between two notions, the
  spectrum and radii of convergence. 
  \subsection*{Acknowledgments}The author wishes to express her gratitude to her
advisors Andrea Pulita and Jérôme Poineau for their precious advice
and suggestions, and for careful reading. She also thank Francesco Baldassarri, Frits
Beukers, Antoine Ducros and Françoise Truc for
useful occasional discussions and suggestions. 

\section{Definitions and notations}\label{sec:defin-an-notat}
We will denote by $\RR$ the field of real numbers, by $\ZZ$ the ring
of integers and by $\NN$ the set of natural numbers. We set
$\R+:=\{r\in \RR;\; r\in \R+\}$.

In all the paper we fix $(k,|.|)$ to be an algebraically closed field
of characteristic $0$ equipped with a trivial valuation. Let $E(k)$ be
the category whose objects are $(\Omega,|.|_\Omega)$, where $\Omega$
is a field extension of $k$, complete with respect to the valuation
$|.|_\Omega$, and whose isomorphisms are isometric rings morphisms.
\subsection*{Analytic affin line and Berkovich spectrum} In order to
get a better visualisation of the spectrum, we shall give an
illustration of the analytic affine line.

Let $\Omega\in E(k)$. We consider $\Omega$-analytic spaces in the sense of Berkovich (see
\cite{Ber}). We denote by $\Ao$ the affine analytic line over the
ground field $\Omega$. 
Recall that a point $x\in \Ao$ corresponds to a multiplicative
semi-norm $|.|_x$ on $\Omega[T]$ (i.e. $|0|_x=0$, $|1|_x=1$,  $|P-Q|_x\leq
\max(|P|_x,|Q|_x)$ and $|P\cdot Q|_x=|P|_x\cdot|Q|_x$ for all $P$, $Q\in k[T]$) whose restriction
coincides with the absolute value of $\Omega$.
\begin{nota}
  Let $x\in \Ao$. We denote by $\Hx$ the complete residue field
  associated to $x$.
\end{nota}

\begin{nota}
  Let $c\in \Omega$ and $r\in \R+$. Denote by $x_{c,r}$ the point that corresponds to the semi-norm (norm
if $r>0$)

\begin{equation}
  \label{eq:1}
  \Fonction{k[T]}{\R+}{\sum_{i=0}^na_i(T-c)^i}{\max_i|a_i|r^i}.
\end{equation}

\end{nota}

\begin{rem}
  Since $(k,|.|)$ is trivially valued, any point of $\Ak$ is a point of the form $x_{c,r}$. The points of the form $x_{c,0}$ coincides with the element of $k$.
\end{rem}

For a point $x_{c,r}\in \Ak$ with $r>0$ the field $\h{x_{c,r}}$ can be
described more concretely. The case where $r=0$ is trivial, indeed
$\h{x_{c,0}}\simeq k$.

In the case where $r<1$, the field $\h{x_{c,r}}$ coincides with the
field of Laurent formal power series

\begin{equation}
  \label{eq:2}
k(\!(T-c)\!):=\left\{\sum_{i\geq N}a_i(T-c)^i;\; a_i\in k,\; N\in
    \ZZ\right\}.
\end{equation}
 equipped with  $(T-c)$-adic
  absolute value given by $|\sum_{i\geq N} a_i(T-c)^i|:=r^N$, if $a_N\ne
  0$.

  If $r>1$, then $\h{x_{c,r}}$ coincides with $k(\!( (T-c)^{-1})\!)$ equipped with the $(T-c)\-1$-adic
  absolute value given by $|\sum_{i\geq N} a_i(T-c)^{-i}|:=r^{-N}$, if $a_N\ne
  0$.

  Otherwise, $\h{x_{c,1}}$ coincides with $k(T)$ equipped with the
  trivial absolute value.
  Let $\Omega\in E(k)$ and $c\in
\Omega$. For $r\in \R+\setminus\{0\}$ we set

\[\discf{\Omega}{c}{r}:=\{x\in \A{\Omega};\; |T(x)-c|\leq r\}\]
and 

\[\disco{\Omega}{c}{r}:=\{x\in \A{\Omega};\; |T(x)-c|< r\}.\]

The point $x_{c,r}\in \A{\Omega}$ corresponds to the disk
$\discf{\Omega}{c}{r}$, more precisely it does not depend on the of
the center $c$ (cf. \cite[Section~1.4.4]{Ber}).\\
Let $c\in k$. The map

\begin{equation}
  \label{eq:3}
  \Fonction{[0,+\infty)}{\Ak}{r}{x_{c,r}}
\end{equation}
induces an homeomorphism between $[0,+\infty)$ and its image. Since
the valuation is trivial on $k$,  we can
describe disks of $\Ak$ as follows.

\begin{nota}
  We denote by $[x_{c,r},\infty)$ (resp. $(x_{c,r},\infty)$) the
  image of $[r,\infty)$ (resp. $(r,\infty)$), by $[x_{c,r}, x_{c,r'}]$
  (resp. $(x_{c,r},x_{c,r'}]$, $[x_{c,r},x_{c,r'})$,
  $(x_{c,r},x_{c,r'})$) the image of $[r,r']$ (resp. $(r,r']$,
  $[r,r')$, $(r,r')$).
\end{nota}

Let $c\in k$ and $r\in \R+$. In the case where $r<1$,  we have

\begin{equation}
  \label{eq:4}
  \discf{k}{c}{r}=[c,x_{c,r}] \qqq \disco{k}{c}{r}=[c,x_{c,r}).
\end{equation}
If $r>1$, recall that for all $a\in k$, $x_{a,1}=x_{0,1}$ and we have 
\begin{equation}
  \label{eq:5}
  \discf{k}{c}{r}=\coprod_{a\in k}[a,x_{0,1})\coprod
  [x_{0,1},x_{0,r}].
\end{equation}
\begin{equation}
  \label{eq:6}
  \disco{k}{c}{r}=
  \begin{cases}
    [c,x_{0,1})& \text{if } r=1\\
    & \\
    \coprod\limits_{a\in k}[a,x_{0,1})\coprod [x_{0,1},x_{0,r})& \text{ otherwise}\\
  \end{cases}.
\end{equation}
\\%
For $r_1$, $r_2\in \R+$, such that $0<r_1\leq r_2$ we set

\[\Couf{\Omega}{c}{r_1}{r_2}:=\discf{\Omega}{c}{r_2}\setminus\disco{\Omega}{c}{r_1}\]
and for $r_1<r_2$ we set:

\[\Couo{\Omega}{c}{r_1}{r_2}:=\disco{\Omega}{c}{r_2}\setminus\discf{\Omega}{c}{r_1}.\]
We may suppress  the index $\Omega$ when it is obvious from the context.

Let $X$ be an affinoid domain of $\A{\Omega}$, we denote by $\cO(X)$ the
$\Omega$-Banach algebra of global sections of $X$.

Since $k$ is trivially valued, if $r<1$ we have

\begin{equation}
  \label{eq:7}
  \fdisf{c}{r}=k[\![T-c]\!]:=\{\sum\limits_{i\in \NN} a_i(T-c)^i;\,
  a_i\in k\},
\end{equation}
otherwise,
\begin{equation}
  \label{eq:8}
  \fdisf{c}{r}=k[T-c].
\end{equation}
In the both cases, the multiplicative norm on $\fdisf{c}{r}$ is :
  \begin{equation}
\label{eq:9}
    \nor{\sum\limits_{i\in \NN}a_i(T-c)^i}=\max_{i\in
  \NN}|a_i|r^i.
\end{equation}
Let $X=\Couf{\Omega}{c}{r_1}{r_2}$ 
\begin{equation}
  \label{eq:10}
\cO(\Couf{\Omega}{c}{r_1}{r_2})=\left\{\sum\limits_{i\in \NN\setminus\{0\}}
  \dfrac{a_{i}}{(T-c)^i};\; a_i\in \Omega,\;|a_{i}|r_1^{-i}\to 0
\right\}\oplus\fdisf{c}{r_2}.
\end{equation}
where
$\nor{\sum\limits_{i\in\NN\setminus\{0\}}\frac{a_{i}}{(T-c)^i}}=\max_i|a_{i}|r_1^{-i}$ and
the sum above is equipped with the maximum norm.\\
\begin{nota}\label{sec:analytic-affin-line}
  Let $X$ be an analytic domain of $\Ak$, and $f\in \cO(X)$. We can
  see $f$ as an analytic map $X\to \Ak$ that we still denote by $f$.
\end{nota}

Recall now the definition of the spectrum introduced by Berkovich.
\begin{Defi}
  Let $E$ be $k$-Banach algebra with unit and $f\in E$.
  The spectrum of $f$ is the set $\Sigma_{f,k}(E)$ of points $x\in
  \Ak$ such that the element $f\ot 1-1\ot T(x)$ is not invertible in
  the $k$-Banach algebra $E\Ct{k}\Hx$.
\end{Defi}
\begin{rem}
  If there is no confusion we denote the spectrum of $f$, as an element of $E$,  just by $\Sigma_f$.
\end{rem}

\begin{rem}
  The set $\Sigma_f\cap k$ coincides with the classical spectrum, i.e.
  \[\Sigma_f\cap k=\{a\in k;\; f-a \text{ is not invertible in } E\}.\]
\end{rem}

\section{Differential modules over $\KS$ and spectra}\label{sec:diff-modul-over}
In this section we recall some properties of differential module
$(M,\nabla)$ over $(F,d)$, where $(F,d)$ is a finite differential
extension of $(\KS,S\d)$. We recall also the definition of the
spectrum of a differential module $(M,\nabla)$ over $(\KS,S\d)$,
introduced in our previous work \cite{Cons}, and explain more
precisely its behavior under ramification of the indeterminate $S$.

 \begin{nota}
 Let $(F,d)$ be a differential field. We denote by $d$-{\bf Mod}($F$) the category of differentiel modules over $(F,d)$ whose arrows are
 morphisms of differential modules.
\end{nota}

\begin{nota}
Let $(F,d)$ be a differential field. We set
$\DD_F:=\Bigoplus_{i\in\NN}F.D^i$ to be the ring of differential
polynomials on $D$ with coefficients in $F$, where the multiplication is
non-commutative and defined as follows: $D.f=d(f)+f.D$ for all $f\in
A$. Let
$P(D)=g_0+\cdots+g_{n-1}D^{n-1}+D^n$ be a monic differential
polynomial. The quotient $\DD_F/\DD_F.P(D)$ is an $F$-vector space of
dimension $n$. Equipped with the multiplication by $D$, it is a
differential module over $(F,d)$.                
\end{nota}

\begin{rem}
 Any differential module over a differential field $(F,d)$, with $d\ne
 0$, is isomorphic to $\DD_F/\DD_F.P(D)$, for some monic differential
 polynomial $P(D)$, thanks to {\it the cyclic vector theorem}. 
\end{rem}

\begin{hyp}
  We fix $r\in (0,1)$ and endow $\KS$ with the $S$-adic
  absolute value given by $|\sum_{i\geq N} a_iS^i|:=r^N$, if $a_N\ne
  0$. In this setting the pair  $(\KS,|.|)$ coincides
  with $\h{x_{0,r}}$, where $x_{0,r}\in \Ak$.
\end{hyp}

\subsection{Spectrum of a differential module}
Recall that if $F$ is a finite extension of $k\((S\))$ of degree $m$,
then we have $F\simeq k\(( S^{\frac{1}{m}}\))$ \cite[Proposition
3.3]{vanga}. The absolute value $|.|$ on $\KS$ extends uniquely to an
absolute value on $F$. The pair $(F,|.|)$ is an element of $E(k)$ and
can be identified with $\h{x_{0,r^{\fra{m}}}}$. Moreover, the
derivation $S\d$ extends uniquely to a bounded derivation $d$ on $F$,
satisfying $d(S^\fra{m})=\fra{m}S^\fra{m}$. Conversely, any finite
differential extension field of $(\KS,S\d)$ is obtained in this way.

We shall recall quickly the construction of the spectrum of a
differential module, which is introduced more precisely in \cite[Section~3.2]{Cons}. Let $(M,\nabla)$ be a differential module over some differential
extension field $(F,d)$ of $(\KS,S\d)$. The $F$-vector space $M$ can
be endowed with a structure of $F$-Banach space, unique up to
bi-bouned isomorphism of $F$-Banach space. In this sitting, the
operator $\nabla$ can be seen as an element of $(\Lk{M},\nor{.}_{\mr{op}})$, the
$k$-Banach algebra of bounded $k$-endomorphisms of $M$ with respect to
the {\it operator norm}. The spectrum
of $(M,\nabla)$ is the spectrum of $\nabla$ as an element of $\Lk{M}$,
that we denote by $\Sigma_{\nabla,k}(\Lk{M})$ (or just
$\Sigma_\nabla$ if it is obvious from the context). This spectrum is a
compact non-empty set, moreover the smallest closed disk centring in
zero containing $\Sigma_{\nabla,k}(\Lk{M})$ has radius equal to
$\nsp{\nabla}=\lim\limits_{n\to
  +\infty}\nor{\nabla}_{\mr{op}}^{\fra{m}}$ (see \cite[Theorem~7.1.2]{Ber}). Stressing also that this
spectrum is invariant by bi-bounded isomorphisms of differential
modules.

However, it depends on the choice of the
derivation. We will see further that $S\d$ has a good behavior under ramification of
the indeterminate, which makes the choice of $S\d$ more
conviniente for the computation of the spectrum.

We need now to recall  importante materials developed in our previous work
\cite{Cons}, that are very necessary to the computation of the
spectrum.

\begin{Pro}
   Let $(M,\nabla)$, $(M_1,\nabla_1)$  and $(M_2,\nabla_2)$ be three
   differential modules over $(F,d)$. If we have two exact
   sequences of the form:
   \[0\to (M_1,\nabla_1)\to (M,\nabla)\to (M_2,\nabla_2)\to 0\]
   \[0\to (M_2,\nabla_2)\to (M,\nabla)\to (M_1,\nabla_1)\to 0.\]
   Then we have $\Sigma_\nabla(\Lk{M})=\Sigma_{\nabla_1}(\Lk{M_1})\cup \Sigma_{\nabla_2}(\Lk{M_2})$.
   
 \end{Pro}

 \begin{proof}
   See \cite[Proposition~3.7]{Cons} and \cite[Remark 3.8]{Cons}.
 \end{proof}

  \begin{cor}\label{sec:spectr-diff-modul}
  Let $(M,\nabla)$, $(M_1,\nabla_1)$  and $(M_2,\nabla_2)$ be three
   differential modules over $(F,d)$. If we suppose that
   $(M,\nabla)=(M_1,\nabla_1)\oplus (M_2,\nabla_2)$, then we have $\Sigma_\nabla(\Lk{M})=\Sigma_{\nabla_1}(\Lk{M_1})\cup \Sigma_{\nabla_2}(\Lk{M_2})$.
 \end{cor}

 \begin{cor}\label{sec:spectr-diff-modul-1}
   Let $f\in F$. Consider the differential module $(M,\nabla):= (\DD_F/\DD_F.(D-f)^n,D)$. Then
   we have:
   \[\Sigma_\nabla(\Lk{M})=\Sigma_{S\d+f}(\Lk{F}).\]
 \end{cor}

 \subsection{Spectrum of a differential module after ramified ground field extension}\label{sec:spectr-diff-module-4}
In this section we show how the spectrum of a differential module behaves after
ramified ground field extensions.

 Let $(F,d)$ be a finite differential extension of $(\KS,S\d)$. Let
$m\in \NN$ such that $F\simeq k(\!( S^{\fra{m}})\!)$. Now, if we set
$Z=S^{\frac{1}{m}}$, then we have $F=k\((Z\))$ and
$d=\frac{Z}{m}\D{Z}$. Note that we can see $(F,\frac{Z}{m}\dz)$ as a
differential module over $(\KS,S\d)$. In the basis $\{1,Z,\cdots,
Z^{m-1}\}$ we have:

\begin{equation}
\label{eq:18}
  \frac{Z}{m}\dz %
    \left(
     \raisebox{0.5\depth}{%
       \xymatrixcolsep{1ex}%
       \xymatrixrowsep{1ex}%
       \xymatrix{f_1\ar@{.}[dddd]\\ \\\\ \\f_{m}\\}%
     }%
   \right)
     =
\left(
     \raisebox{0.5\depth}{%
       \xymatrixcolsep{1ex}%
       \xymatrixrowsep{1ex}%
     \xymatrix{S\d
       f_1\ar@{.}[ddd]\\ \\ \\S\d f_{m}\\}%
}
   \right)
+
\left(
     \raisebox{0.5\depth}{%
       \xymatrixcolsep{1ex}%
       \xymatrixrowsep{1ex}%
   \xymatrix{0& 0\ar@{.}[rr]\ar@{.}[ddrr] &    &0\ar@{.}[dd]\\
                    0\ar@{.}[ddrr]\ar@{.}[dd]&\fra{m}\ar@{.}[ddrr]                 &    & \\
                      &                   &    &0\\
                    0\ar@{.}[rr]&                   &0  &\frac{m-1}{m}\\}
}
   \right)
   \left(
     \raisebox{0.5\depth}{%
       \xymatrixcolsep{1ex}%
       \xymatrixrowsep{1ex}%
       \xymatrix{f_1\ar@{.}[dddd]\\ \\\\ \\f_{m}\\}%
     }
   \right)
 \end{equation}
 We have a functor:
 \begin{equation}
   \label{eq:19}
   \Fonction{{I_F}^*:S\d-\text{\bf Mod}(\KS)}{\frac{Z}{m}\dz-\text{\bf Mod}(F)}{(M,\nabla)}{({I_F}^*M,{I_F}^*\nabla)}
 \end{equation}
 where ${I_F}^*M=M\ot_{\KS}F$ and the connection ${I_F}^*\nabla$ is defined as follows:
 \[{I_F}^*\nabla=\nabla\ot 1+1\ot \frac{Z}{m}\dz.\]
Let $(M,\nabla)$ be an object of $\diffk$ of rank $n$. If
$\{e_1,\dots,e_n\}$ is a basis of $M$ such that we have:
 \[\nabla \begin{pmatrix} f_1\\ \vdots \\ f_n \end{pmatrix}
		= \begin{pmatrix}S\d f_1\\ \vdots \\ S\d f_n \end{pmatrix}+G 
		\begin{pmatrix} f_1\\ \vdots \\ f_n \end{pmatrix},\]
with $G\in\cM(\KS)$, then $({I_F}^*M,{I_F}^*\nabla)$ is of rank $n$ and in the basis $\{e_1\ot 1,\dots, e_n\ot 1\}$
we have:
\begin{equation}
  \label{eq:20}
  {I_F}^*\nabla \begin{pmatrix} f_1\\ \vdots \\ f_n \end{pmatrix}
		= \begin{pmatrix}\frac{Z}{m}\dz f_1\\ \vdots \\ \frac{Z}{m}\dz f_n \end{pmatrix}+G 
		\begin{pmatrix} f_1\\ \vdots \\ f_n \end{pmatrix}.
              \end{equation}
 We have also the functor:
 \begin{equation}
\label{eq:21}
   \Fonction{I_{F*}: \frac{Z}{m}\dz-\text{\bf
     Mod}(F)}{ S\d-\text{\bf Mod}(\KS)} {(M,\nabla)}{(I_{F*}M,I_{F*}\nabla)}
\end{equation}
where $I_{F*} M$ is the restriction of scalars of $M$ via
$\KS\hookrightarrow F$, and $\nabla=I_{F*}\nabla$ are equal as
$k$-linear maps. If $(M,\nabla)$ has rank equal to $n$, the rank of
$(I_{F*}M,I_{F*}\nabla)$ is equal to $n.m$.

Let $(M,\nabla)$ be an object of $S\d-\text{\bf Mod}(\KS)$ of rank
$n$. The differential module
$(I_{F*}{I_F}^*M,I_{F*}{I_F}^*\nabla)$ has rank
$nm$. Let $\{e_1,\cdots, e_n\}$ be a basis of $(M,\nabla)$ and let $G$
be the associated matrix in this basis. Then the matrix associated to
$(I_{F*}{I_F}^*M,I_{F*}{I_F}^*\nabla)$ in the basis
$\{e_1\ot 1,\cdots ,e_n\ot 1, e_1\ot Z,\cdots, e_n\ot Z,\cdots, e_1\ot
Z^{m-1},\cdots, e_n\ot Z^{m-1}\}$ is:
\[\left(
     \raisebox{0.5\depth}{%
       \xymatrixcolsep{1ex}%
       \xymatrixrowsep{1ex}%
       \xymatrix{G& 0\ar@{.}[rr]\ar@{.}[ddrr] &    &0\ar@{.}[dd]\\
                    0\ar@{.}[ddrr]\ar@{.}[dd]&G+\fra{m} \cdot I_n\ar@{.}[ddrr]                 &    & \\
                      &                   &    &0\\
                    0\ar@{.}[rr]&                   &0
                    &G+\frac{m-1}{m}\cdot I_n\\}
}
\right).\]
Therefore we have the following isomorphism:

 \begin{equation}
\label{eq:28}
   (I_{F*}{I_F}^*M,I_{F*}{I_F}^*\nabla)\simeq\Bigoplus_{i=0}^{m-1}(M,\nabla+\frac{i}{m}).
 \end{equation}
 As $k$-Banach spaces $I_{F*}{I_F}^*M$ and ${I_F}^*M$ are the same, and
 $I_{F*}{I_F}^*\nabla$ as a $k$-linear map coincides with
 ${I_F}^*\nabla$.  Therefore, we have
 \begin{equation}
   \label{eq:23}
   \Sigma_{I_F^*\nabla,k}(\Lk{I_F^*M})=\Sigma_{I_{F*}I_F^*\nabla,k}(\Lk{I_{F*}I_F^*M}).
 \end{equation}
 By \cite[Proposition7.1.4]{Ber} and Corollary~\ref{sec:spectr-diff-modul} we have:
 \begin{equation}
   \label{eq:22}
   \Sigma_{{I_F}^*\nabla,k}(\Lk{{I_F}^*M})=\Bigcup_{i=0}^{m-1}\frac{i}{m}+\Sigma_{\nabla,k}(\Lk{M})\footnotemark.
 \end{equation}
 \footnotetext{This is the image of the spectrum by the polynomial
   function $\frac{i}{m}+T$.}

 \subsection{Newton polygon and the decomposition according to the slopes}
 Let $v:k\((S^{\fra{m}}\))\to \ZZ\cup\{\infty\}$ be the valuation map
 associated to the $S^{\fra{m}}$-adic valuation.
 
 Let $P=\Sum_{i=0}^{n}g_iD^i$ be an element of $\DD_{\KS}$. Let $L_P$ be
 the convex hull in $\RR^2$ of the set of points \[\{(i,v(g_i))|\;
 0\leq i\leq n\}\cup \{(0,\Min_{0\leq i\leq n} v(g_i))\}.\]
 
 \begin{Defi}[{\autocite[Definition 3.44]{vanga}}]
  The Newton polygon $\mathrm{NP}(P)$ of $P$ is the boundary of $L_P$. The
  finite slopes $\gamma_i$ of $P$ are called the slopes of $\mathrm{NP}(P)$. The horizontal width of the segment of $NP(P)$
  of slope $\gamma_i$ is called the
  multiplicity of $\gamma_i$.
\end{Defi}
\begin{Defi}
   A differential module $(M,\nabla)$ over $(\KS,S\d)$ is said to be
   regular singular if there exists a basis for which the associated matrix $G$
   has constant entries (i.e $G\in\cM_n(k)$).
                We will call the eigenvalues of such $G$ the exponents of
                $(M,\nabla)$. 
              \end{Defi}

 \begin{Pro}
  Let $(M,\nabla)$ be a differential module over $(\KS,S\d)$. The following
  properties are equivalent:
  \begin{itemize}
  \item $(M,\nabla)$ is regular singular;
  \item There exists a differential polynomial $P(D)$ with only one
    slope equal to $0$ such that $(M,\nabla)\simeq
    \DD_{\KS}/\DD_{\KS}.P(D)$;
  \item There exists $P(D)=g_0+g_1D+\cdots+g_{n-1}D^{n-1}+D^n$ with 
     $g_i\in \ks$, such that $(M,\nabla)\simeq (\DD_{\KS}/\DD_{\KS}.P(D),D)$;
  \item There exists $P(D)=g_0+g_1D+\cdots+g_{n-1}D^{n-1}+D^n$ with 
     $g_i\in k$, such that  $(M,\nabla)\simeq (\DD_{\KS}/\DD_{\KS}.P(D),D);$

  \end{itemize}
\end{Pro}

\begin{proof}
  See \cite[Corollary~7.1.3]{Ked} and \cite[Proposition~10.1]{christol}.
\end{proof}

\begin{Pro}[{\cite[Proposition~7.3.6]{Ked}}]\label{sec:spectr-regul-sing}
  Let $P(D)=g_0+g_1D+\cdots+g_{n-1}D^{n-1}+D^n$ such that
  $g_i\in\ks$. Then we have the isomorphism in $\diffk$:
  \[(\DD_{\KS}/\DD_{\KS}.P(D),D)\simeq (\DD_{\KS}/\DD_{\KS}.P_0(D)),\] where
  $P_0(D)=g_0(0)+g_1(0)D+\cdots g_{n-1}(0)D^{n-1}+D^n$. 
  
\end{Pro}

\begin{rem}\label{sec:newt-polyg-decomp-1}

  This proposition means in
  particular that for all $f\in \ks$ there exists
  $g\in\KS\setminus\{0\}$ such that $f-\frac{S\d(g)}{g}=f(0)$. Indeed,
  by Proposition~\ref{sec:spectr-regul-sing} we have $(\KS,
  S\d+f)\simeq (\KS,S\d+f(0))$. This is equivalente to saying that there
  exists $g\in \KS\setminus\{0\}$ such that
  \[g\-1\circ S\d\circ g+f=S\d+f(0).\] 
\end{rem}

\begin{Defi}
 Let $(M,\nabla)$ be an element of $\diffk$, and let
 $P(D)\in\DD_{\KS}$ such that $(M,\nabla)\simeq (\DD_{\KS}/\DD_{\KS}.P(D),D)$. If all the
 slopes of $P(D)$ are different from $0$, then we say that
 $(M,\nabla)$ is without regular part.  
\end{Defi}

\begin{Pro}\label{sec:newt-polyg-decomp}
  Let $(M,\nabla)$ be a differential module over $(\KS,S\d)$. Then we have decomposition in $\diffk$:
  \[ (M,\nabla) =(M_{\mr{reg}},\nabla_{\mr{reg}})\oplus
    (M_{\mr{irr}},\nabla_{\mr{irr}})\]
 where $(M_{\mr{reg}},\nabla_{\mr{reg}})$ is a regular singular
 differential module  and
 $(M_{\mr{irr}},\nabla_{\mr{irr}})$ is a
 differential module without regular part.
\end{Pro}

\begin{proof}
  See \cite[Proposition 12.1]{christol} and \cite[Theorem~3.48]{vanga}.
\end{proof}

\section{Spectrum of a differential module}\label{sec:spectr-diff-module-5}

In this section we compute the spectrum of a differential module
$(M,\nabla)$ over \\ $(\KS,S\d)$. The
main statement of the paper is the following.
\begin{Theo}\label{sec:spectr-diff-module-3}
  Let $r\in (0,1)$. Assume that $(\KS,|.|)\simeq (\h{x_{0,r}},|.|)$. Let $(M,\nabla)$ be a differential module over $(\KS,S\d)$. Let
  $\{\gamma_1,\cdots, \gamma_{\nu}\}$ be the set of slopes of
  $(M,\nabla)$ and let
  $\{a_1,\cdots, a_{\mu}\}$ be the set of exponents of the regular
  part of $(M,\nabla)$. Then the spectrum of $\nabla$ as an element of
  $\Lk{M}$ is:
  \[\Sigma_{\nabla,k}(\Lk{M})=\{x_{0,r^{-\gamma_1}},\cdots,
    x_{0,r^{-\gamma_{\nu}}}\}\cup \bigcup_{i=1}^{\mu}(a_i+\ZZ)\]
\end{Theo}

\begin{rem}
  We observe that although differential modules over $(\KS,S\d)$
  are algebraic objects, their spectra in the sense of Berkovich depends highly on the choice of
  the absolute value on $\KS$.  
\end{rem}

According to
Corollary~\ref{sec:newt-polyg-decomp} we have the decomposition:
\[ (M,\nabla) =(M_{\text{reg}},\nabla_{\text{reg}})\oplus
  (M_{\text{irr}},\nabla_{\text{irr}})\]
We know that $\Sigma_\nabla=\Sigma_{\nabla_{\text{reg}}}\cup
\Sigma_{\nabla_{\text{irr}}}$
(cf. Proposition~\ref{sec:spectr-diff-modul}). Therefore, in order to obtain the
main statement, it is enough to know the spectrum of a regular singular differential
module and the spectrum of differential
module without regular part.

\subsection{Spectrum of regular singular differential module}\label{sec:spectr-regul-sing-3}

Let $(M,\nabla)$ be a regular singular differential module. We point
out, since our work in \cite{Cons},  that the difficulty of the
computation of the spectrum of $\nabla$ is reduced to the computation of $\Sigma_{S\d,k}(\Lk{\KS})$. This is due, more precisely, to the following
Proposition.
\begin{Pro}[{\cite[Proposition~3.15]{Cons}}]\label{sec:spectr-regul-sing-1}
         Let $(M,\nabla)$ be a differential module over $(\KS,S\d)$
                such that:
		\[\nabla \begin{pmatrix} f_1\\ \vdots \\ f_n \end{pmatrix}
		= \begin{pmatrix} df_1\\ \vdots \\ df_n \end{pmatrix}+G 
		\begin{pmatrix} f_1\\ \vdots \\ f_n \end{pmatrix},\]
		with $G\in \cM_n(k)$. The spectrum of $\nabla$ is
                $\Sigma_{\nabla,k}(\Lk{M})=\Bigcup_{i=1}^{N}(a_i+\Sigma_{S\d,k}(\Lk{\KS}))$, where
                $\{a_1,\dots, a_N\}$ are  the eigenvalues of $G$.
               
              \end{Pro}%
The
behaviour of the spectrum of $\nabla$ is recapitulated in the
following theorem.
              
\begin{Theo}\label{sec:spectr-diff-module}
  Let $(M,\nabla)$ be a regular singular differential module over
  $(\KS,S\d)$. Let $G$ a matrix associated to $\nabla$ with constant
  entries (i.e. $G\in\cM_\nu(k)$), and let $\{a_1,\cdots, a_N\}$ be
the set of eigenvalues of $G$. The spectrum of $\nabla$ is
                \[\Sigma_{\nabla,k}(\Lk{M})=\Bigcup_{i=1}^{N}(a_i+\ZZ)\cup\{x_{0,1}\}.\]
               
              \end{Theo}
              \begin{Lem}[{\cite[Proposition~3.12]{vanga}}]
                There exists a basis for which the set of the eigenvalues
                $\{a_1,\cdots, a_N\}$ of $G$ satisfies $a_i-a_j\not \in\ZZ$ for each $i\ne j$.
              \end{Lem}
              We now compute the spectrum of $S\d$.

              \begin{Lem}\label{sec:spectr-regul-sing-2}
  The norm and spectral semi-norm of
  $S\d$ as an element of $\Lk{\KS}$ satisfy:
    \[\nor{S\d}=1, \qq \nsp{S\d}=1.\]
  \end{Lem}

  \begin{proof}
    Since $\nor{S}=|S|=r$ and $\nor{\d}=\fra{r}$ (cf. \cite[Lemma~4.4.1]{and}),
    we have $\nor{S\d}\leq 1$. Hence also, $\nsp{S\d}\leq 1$. The map
    \[\Fonction{\Lk{\Hx}}{\Lk{\Hx}}{\phi}{S\-1\circ\phi\circ S}\]
    is bi-bounded and induces change of basis. Therefore, 
    as
    $S\-1\circ(S\d)\circ S=S\d+1$, we have $\nsp{S\d}=\nsp{S\d+1}$. Since $1$
    commutes with $S\d$, we have:
    \[1=\nsp{1}=\nsp{S\d+1-S\d}\leq\max(\nsp{S\d+1},\nsp{S\d}).\]
  Consequently, we obtain
    
    \[\nor{S\d}=\nsp{S\d}=1.\]   
  \end{proof}
  \begin{Pro}\label{sec:spectr-line-diff-1}
  The spectrum of $S\d$ as an element of $\Lk{k\((  S\))}$ is equal to:
  \[\Sigma_{S\d}(\Lk{k\(( S\)) })=\ZZ\cup \{x_{0,1}\}\]
\end{Pro}
  
  \begin{proof}
    We set $d:=S\d$ and $\Sigma_{d-n}:=\Sigma_{d-n,k}(\Lk{\KS})$. As $\nsp{d}=1$
    (cf. Lemma~\ref{sec:spectr-regul-sing-2}), we have
    $\Sigma_d\subset\disf{0}{1}$.

    Let $a\in k\cap\disf{0}{1}$. If $a\in \ZZ$,  then we have
    $(d-a)(S^a)=0$. Hence, $d-a$ is not
    injective and $\ZZ\subset \Sigma_d$. As the spectrum is
    compact, we have $\ZZ\cup\{x_{0,1}\}\subset \Sigma_d$. If
    $a\not\in \ZZ$, then $d-a$ is invertible in $\Lk{\KS}$. Indeed, let
    $g(S)=\sum_{i\in \ZZ}b_i S^i\in \KS$. If there exists
    $f=\sum_{i\in \ZZ}a_i S^i\in \KS$ such that $(d-a)f=g$, then for
    each $i\in \ZZ$ we have
    \[a_i= \frac{b_i}{(i-a)}.\]
  For each
    $i\in\ZZ$ we have $|a_i|=|b_i|$. This means that $f$ it is unique
    and converges in $\KS$. We obtain also 
    $|f|= |g|$. Consequently, the set theoretical inverse
    $(d-a)^{-1}$ is bounded and $\nor{(d-a)\-1}=1$. Hence, we have
    $\nsp{(d-a)\-1}=1$. According to
    \cite[Lemma~2.20]{Cons}, we have $\diso{a}{1}\subset
    \Ak\setminus\Sigma_d$.

    Recall that $\disf{0}{1}=\bigcup_{a\in
      k}[a,x_{0,1}]$ (cf. \eqref{eq:4}). In order to end the proof, it is enough to show that
    $(n,x_{0,1})\subset\Ak\setminus \Sigma_d$ for all $n\in \ZZ$. Let
    $n\in \ZZ$. Then we have 

    \[\KS=k.S^n\oplus
      \widehat{\bigoplus}_{i\in\ZZ\setminus\{n\}}k.S^i.\]
The operator $(d-n)$
      stabilises both $k.S^n$ and
      $\widehat{\bigoplus}_{i\in\ZZ\setminus\{n\}}k.S^i$. We set
      $(d-n)|_{k.S^n}=\nabla_1$ and
      $(d-n)|_{\widehat{\bigoplus}_{i\in\ZZ\setminus\{n\}}k.S^i}=\nabla_2$. We
      set  $\Sigma_{\nabla_1}:=\Sigma_{\nabla_1,k}(\Lk{k.S^n})$ and
      $\Sigma_{\nabla_2}:=\Sigma_{\nabla_2,k}(\Lk{\widehat{\bigoplus}_{i\in\ZZ\setminus\{n\}}k.S^i})$. We
      have $\nabla_1=0$. By \cite[Lemma~2.29]{Cons}, we have:
      \[\Sigma_{d-n}=\Sigma_{\nabla_1}\cup\Sigma_{\nabla_2}=\{0\}\cup
        \Sigma_{\nabla_2}.\]
      We now prove that
      \[\diso{0}{1}\cap\Sigma_{\nabla_2}=\emptyset.\]
      The operator $\nabla_2$ is invertible in
      $\Lk{\widehat{\bigoplus}_{i\in\ZZ\setminus\{n\}}k.S^i}$. Indeed,
      let $g(S)=\sum_{i\in \ZZ\setminus\{n\}}b_i S^i\in \widehat{\bigoplus}_{i\in\ZZ\setminus\{n\}}k.S^i$. If there exists
    $f=\sum_{i\in \ZZ\setminus\{n\}}a_i S^i\in \widehat{\bigoplus}_{i\in\ZZ\setminus\{n\}}k.S^i$ such that $\nabla_2(f)=g$ , then for
    each $i\in \ZZ\setminus\{n\}$ we have
    \[a_i= \frac{b_i}{(i-n)}.\]
    Since $|a_i|=|b_i|$, the element $f$ exists and it is unique,
    moreover $|f|=|g|$. Hence, $\nabla_2$ is invertible in
    $\Lk{\widehat{\bigoplus}_{i\in\ZZ\setminus\{n\}}k.S^i}$ and as a
    $k$-linear map it is isometric. Therefore, we have
    $\nsp{\nabla_2\-1}=1$. Hence, by
 \cite[Lemma~2.20]{Cons}
    $\diso{0}{1}\subset\Ak\setminus\Sigma_{\nabla_2}$. Consequently,
    $\diso{0}{1}\cap \Sigma_{d-n}=\{0\}$. As
    $\Sigma_d=\Sigma_{d-n}+n$
      (cf. \cite[Proposition7.1.4]{Ber}), we have
      $\diso{n}{1}\cap\Sigma_d=\{n\}$. Therefore, for all $n\in
      \ZZ$ we have
      $(n,x_{0,1})\subset\Ak\setminus \Sigma_d$ and the claim follows. 
  \end{proof}

\subsection{Spectrum of a differential module without regular part}
              
Recall the following Theorem, which is the celebrated theorem of Turrittin. It ensures that any differential module becomes extension of rank one differential modules after pull-back by a suitable ramified extension. 

\begin{Theo}[\cite{Tur55}]\label{sec:spectr-pure-irreg}
  Let $(M,\nabla)$ be a differential module over $(\KS,S\d)$. There exists
  a finite extension $F=k\((S^{\fra{m}}\))$ such that we have:

  \begin{equation}
   \label{eq:24}
    ({I_F}^* M, {I_F}^*\nabla)=\Bigoplus_{i=1}^N (\DD_F/\DD_F.(D-f_i)^{\alpha_i},D)    
  \end{equation}
where $f_i\in k[\![S^{-\fra{m}}]\!]$ and $\alpha_i \in\NN$.
\end{Theo}

\begin{proof}
  See \cite[Theorem 3.1]{vanga}.
\end{proof}

Now, in order to compute the spectrum, we need the following
result.

\begin{Pro}\label{sec:spectr-line-diff-4}
  Let $f=\sum_{i\in\ZZ}a_iS^{\frac{i}{m}}$ an element of $F:=k\(( S^{\fra{m}}\)) $ and let $(F,\nabla)$ be the differential module of
  rank one such that $\nabla=S\d+f$.
  If $v(f)<0$, then the spectrum of
  $\nabla$ as an element of $\Lk{F }$ is:
  \[\Sigma_{\nabla,k}(\Lk{F})=\{x_{0,r^{v(f)}}\}.\]
\end{Pro}

The following results are necessary to prove this proposition.

             \begin{Lem}\label{sec:spectr-vers-youngs}
                Let $\Omega\in E(k)$. Consider the isometric embedding of $k$-algebras
                \[\Fonction{\Omega}{\Lk{\Omega}}{a}{b\mapsto a.b}.\]
                With respect to this embedding, $\Omega$ is a maximal
                commutative subalgebra of $\Lk{\Omega}$.
              \end{Lem}
              \begin{proof}
  Let $A$ be a commutative subalgebra of $\Lk{\Omega}$ such that
  $\Omega\subset A$. Then each element of  $A$ is an endomorphism of
  $\Omega$ that  commutes with the
  elements of $\Omega$. Therefore,
  $A\subset\LL{\Omega}{\Omega}=\Omega$. Hence, we have $A=\Omega$.
\end{proof}

\begin{Lem}\label{sec:spectr-line-diff-3}
 Let $\Omega\in E(k)$ and $\pik{\Omega}:\A{\Omega}\to \Ak$ be the
 canonical projection. Let $\alpha\in\Omega$. The spectrum of $\alpha$ as an element of $\Lk{\Omega}$ is
  $\Sigma_\alpha(\Lk{\Omega})=\{\pik{\Omega}(\alpha)\}$.   
\end{Lem}
\begin{proof}
By \cite[Proposition~7.1.4, i)]{Ber}, the spectrum of $\alpha$ as an element of $\Omega$ is
 the point which corresponds to the character $k[T]\to \Omega$,
 $T\mapsto \alpha$. Hence, $\Sigma_{\alpha,k}(\Omega)=\{\pik{\Omega}(\alpha)\}$. By
 Lemma~\ref{sec:spectr-vers-youngs} and \cite[Proposition~7.2.4]{Ber} we conclude. 
\end{proof}

\begin{proof}[Proof of Proposition~\ref{sec:spectr-line-diff-4}]
  We set $d:=S\d$. We can assume that $f=\sum_{i\in\NN}a_iS^{\frac{-i}{m}}$. Indeed,
  since $f=f_-+f_+$ with $f_-:=\sum_{i<0}a_iS^{\frac{-i}{m}}$ and
  $f_+:=\sum_{i\geq 0}a_iS^{\frac{-i}{m}}$, according to
  Remark~\ref{sec:newt-polyg-decomp-1} there exists $g\in k\(( S^{\fra{m}}\))$ such that
  $f_+-a_0=\frac{S\d(g)}{g}$. Therefore, we have $(k\((
  S^{\fra{m}}\)),\nabla)\simeq (k\((
  S^{\fra{m}}\)),S\d+f_-+a_0)$. Since the point $\pik{F}(f)$ corresponds to the
  character $k[T]\to F$ , $T\mapsto f$ and $F\simeq \h{x_{0,r^{\fra{m}}}}$, it coincides with
  $f(x_{0,r^{\fra{m}}})$(cf. Notation~\ref{sec:analytic-affin-line}). Moreover,
  we have $f(x_{0,r^{\fra{m}}})=x_{0,|f|}=x_{0,r^{v(f)}}$. By Lemma~\ref{sec:spectr-line-diff-3}
   $\Sigma_{f,k}(\Lk{F})=\{x_{0,r^{v(f)}}\}$. Let us prove now
 that $\Sigma_{\nabla,k}(\Lk{F})=\{x_{0,r^{v(f)}}\}$. Let
 $y\in\Ak\setminus\{x_{0,r^{v(f)}}\}$. We know that $f\ot 1-1 \ot
 T(y)$ is invertible in $F\ct_k\Hy$, hence invertible in $\Lk{F}\ct_k\Hy$. Since $d\ot
 1=(\nabla\ot 1-1\ot T(y))-(f\ot 1 -1\ot T(y))$, in order to prove
 that $\nabla\ot 1-1\ot T(y)$ is invertible, it is enough to show that
 \[\nor{d\ot 1}<\nor{(f\ot 1-
      1\ot T(y))\-1}\-1.\]
In order to do so, since
  $\nor{d}=\nor{d\ot 1}=1$ (cf. \cite[Lemma~2.3]{Cons} and Lemma \ref{sec:spectr-regul-sing-2} ), it is enough to
  show that $1<\nor{(f\ot 1-
      1\ot T(y))\-1}\-1$. On the one hand, since $F
  \hookrightarrow \Lk{F}$ is an isometric embedding, then so is
  $F\ct_k \Hy\to \Lk{F}\ct_k \Hy$ (cf. \cite[Lemme 3.1]{poi}). On
  the other hand, we have $F\ct_k \Hy=\fCouf{\Hy}{0}{r^{\fra{m}}}{r^{\fra{m}}}$. Therefore, we have
  \[\nor{(f\ot 1-
      1\ot T(y))\-1}\-1=\nor{f\ot 1- 1\ot
      T(y)}=\max(|f-a_0|,|T(y)-a_0|).\]
  Consequently, we obtain $1<\nor{(f\ot 1-
      1\ot T(y))\-1}\-1 $ in $F\ct_k\Hy$, hence in
    $\Lk{F}\ct_k\Hy$. Since
    the spectrum $\Sigma_{\nabla,k}(\Lk{F})$ is not empty, we conclude
    that $\Sigma_{\nabla,k}(\Lk{F})=\{x_{0,r^{v(f)}}\}$.
\end{proof}

\begin{Pro}\label{sec:spectr-diff-module-1}
  Let $(M,\nabla)$ be a differential module
  over $(\KS,S\d)$ without regular part. The spectrum of $\nabla$ as an element of $\Lk{M}$ is:
  \[\Sigma_{\nabla,k}(\Lk{M})=\{x_{0,r^{v(f_1)}},\cdots, x_{0,r^{v(f_N)}}\}\]
where the $f_i$ are as in the formula \eqref{eq:24}.
\end{Pro}

\begin{proof}
We set $\Sigma_\nabla:=\Sigma_{\nabla,k}(\Lk{M})$. By Theorem~\ref{sec:spectr-pure-irreg}, there exists
$F=k\((S^{\fra{m}}\))$ such

\[({I_F}^* M, {I_F}^*\nabla)=\Bigoplus_{i=1}^N
  \DD_F/\DD_F.(D-f_i)^{\alpha_i} \]
where $f_i\in k[\![S^{-\fra{m}}]\!]$. We set
$\Sigma_{{I_F}^*\nabla}:=\Sigma_{{I_F}^*\nabla,k}(\Lk{{I_F}^*M})$. Since
$(M,\nabla)$ is without regular part, we have $f_i\in k[\![S^{-\fra{m}}]\!]\setminus k
                                                                                $. By
Corollaries ~\ref{sec:spectr-diff-modul} and \ref{sec:spectr-diff-modul-1}, we have:

\[\Sigma_{{I_F}^*\nabla}=\Bigcup_{i=1}^N\Sigma_{S\d+f_i}(\Lk{F}).\]
By Proposition~\ref{sec:spectr-line-diff-4}, we have
$\Sigma_{S\d+f_i}(\Lk{F})=\{x_{0,r^{v(f_i)}}\}$. Hence,
\[\Sigma_{{I_F}^*\nabla}=\{x_{0,r^{v(f_1)}},\cdots, x_{0,r^{v(f_N)}}\}.\]
By the formula~\eqref{eq:22}, we have:

\[\Sigma_{{I_F}^*\nabla}=\Bigcup_{i=0}^{m-1}\frac{i}{m}+\Sigma_\nabla.\]
Since $r^{v(f_i)}>1$
for all $1\leq i\leq N$, then each element of
$\Sigma_{{I_F}^*\nabla}$ is invariant by translation by $\frac{j}{m}$
where $1\leq j\leq m$. This means that
$\Sigma_\nabla=\Sigma_\nabla+\frac{j}{m}$. Therefore, we have $\Sigma_{{I_F}^*\nabla}=\Sigma_\nabla$. 
\end{proof}

\begin{rem}\label{sec:spectr-diff-module-2}
  Note that, it is not easy to compute the $f_i$ of the
  formula~\eqref{eq:24}. However, the values $-v(f_i)$ coincide with
  the slopes of the differential module (cf. \cite{Kat87} and \cite[Remarks~3.55]{vanga}).
\end{rem}

We now prove the main statement of the paper that summarizes all the previous
results.

\begin{proof}[Proof of Theorem~\ref{sec:spectr-diff-module-3}]
  According to Theorem~\ref{sec:spectr-diff-module},
  Proposition~\ref{sec:spectr-diff-module-1} and
  Remark~\ref{sec:spectr-diff-module-2} we obtain the result.
\end{proof}

\begin{rem}
  We note that, although differential modules over $(\KS,S\d)$ are algebraic objects, their spectra in
the sense of Berkovich depend highly on the choice of the absolute
value on $\KS$. Indeed, as en exemple we chose
$(\DD_{\KS}/\DD_{\KS}\cdot (D-S^{-2}),D)$ and let $r$, $r'\in (0,1)$,
with $r\ne r'$. As fields
$\KS$, $\h{x_{0,r}}$, $\h{x_{0,r'}}$ are the same. However, we have
\[\Sigma_D(\Lk{\h{x_{0,r}}})=\{x_{0,r^{-2}}\}\ne \{x_{0,r'^{-2}}\}=\Sigma_D(\Lk{\h{x_{0,r'}}})\]
\end{rem}

\section{Conclusion}
  We may think that this result does not give more information than
  decomposition theorems. However, it shows clearly the existence of a
  link between the spectrum and the slopes of a differential
  module. Consequently, it leads to a connection between the spectrum
  and all radii of convergence.

  More precisely, assume that $(\KS,|.|)$ coincides
  with $\h{x_{0,r}}$ for some $r\in (0,1)$ and let $(M,\nabla)$ be a
  differential module over $(\KS,S\d)$. Let $\gamma_1<\gamma_2<\cdots
  <\gamma_\nu$ be the positive formal slopes of $(M,\nabla)$. Then the
  radii of
  convergence of the solutions of
  $(M_{\text{irr}},\nabla_{\text{irr}})$ are exactly:
  \[R_i= r^{\gamma_i};\; i\in\{1,\cdots,\nu\}\]
  (cf. \cite[Proposition~4.3.1]{and} and \cite[Section~4, (4.4)]{and}) and satisfy
  $R_1<R_2<\cdots<R_\nu$. Consider now the spectrum
  $\Sigma_\nabla$ of $(M,\nabla)$. If $
  \Sigma_\nabla\setminus k\cup\{x_{0,1}\}=\{x_{0,r_1},\cdots,x_{0,r_\nu}\}$, with $r_\nu<\cdots <
    r_0$,  then we can deduce
    directly that the
  radii of
  convergence of the solutions of the irregular part of
  $(M,\nabla)$ are  $R_i=\fra{r_i}$ for
  $i\in\{1,\cdots, \nu\}$ and we have:
  \[ R_1<R_2<\cdots <R_\nu\]

  This motivates to push the study further, and get an analogous result for
  a more general context. Notably, for the $p$-adic
  case, where the existens of $p$-adic Liouville number causes many
  problems namely the infinitude of de Rah cohomology. We hope that,
  using this spectrum, we
  can aviode the non-Liouville assumption. However, this notion of
  spectrum need to be more refined. Indeed, the multiplicity is
  clearly  a
  skipped information.

\printbibliography
Tinhinane Amina, AZZOUZ

\end{document}